\documentclass[draft]{amsart} 
\usepackage{amssymb,latexsym,amsfonts,verbatim,amscd,ifthen, mathrsfs} 
\usepackage{color}
\usepackage{url}

\numberwithin{equation}{section}

\theoremstyle{plain}

\newtheorem{theorem}[equation]{Theorem}   
\newtheorem{lemma}[equation]{Lemma} 
 
\newtheorem{corollary}[equation]{Corollary} 
\newtheorem{observation}[equation]{Observation}

\theoremstyle{definition}

\newtheorem{definition}[equation]{Definition} 
\newtheorem{remark}[equation]{Remark}

\DeclareMathOperator{\id}{id}

\begin{document}   

\renewcommand{\:}{\! :} 
\newcommand{\p}{\mathfrak p} 
\newcommand{\m}{\mathfrak m}
\newcommand{\e}{\epsilon}
\newcommand{\lra}{\longrightarrow}
\newcommand{\lla}{\longleftarrow} 
\newcommand{\ra}{\rightarrow} 
\newcommand{\altref}[1]{{\upshape(\ref{#1})}} 
\newcommand{\bfa}{\boldsymbol{\alpha}} 
\newcommand{\bfb}{\boldsymbol{\beta}} 
\newcommand{\bfg}{\boldsymbol{\gamma}} 
\newcommand{\bfM}{\mathbf M} 
\newcommand{\bfI}{\mathbf I} 
\newcommand{\bfC}{\mathbf C} 
\newcommand{\bfB}{\mathbf B} 
\newcommand{\bsfC}{\bold{\mathsf C}} 
\newcommand{\bsfT}{\bold{\mathsf T}}
\newcommand{\smsm}{\smallsetminus} 
\newcommand{\ol}{\overline} 
\newcommand{\os}{\overset}

\newlength{\wdtha}
\newlength{\wdthb}
\newlength{\wdthc}
\newlength{\wdthd}
\newcommand{\elabel}[1]
           {\label{#1}  
            \setlength{\wdtha}{.4\marginparwidth}
            \settowidth{\wdthb}{\tt\small{#1}} 
            \addtolength{\wdthb}{\wdtha}
            \raisebox{\baselineskip}
            {\color{red} 
             \hspace*{-\wdthb}\tt\small{#1}\hspace{\wdtha}}}  

\newcommand{\mlabel}[1] 
           {\label{#1} 
            \setlength{\wdtha}{\textwidth}
            \setlength{\wdthb}{\wdtha} 
            \addtolength{\wdthb}{\marginparsep} 
            \addtolength{\wdthb}{\marginparwidth}
            \setlength{\wdthc}{\marginparwidth}
            \setlength{\wdthd}{\marginparsep}
            \addtolength{\wdtha}{2\wdthc}
            \addtolength{\wdtha}{2\marginparsep} 
            \setlength{\marginparwidth}{\wdtha}
            \setlength{\marginparsep}{-\wdthb} 
            \setlength{\wdtha}{\wdthc} 
            \addtolength{\wdtha}{1.1ex} 
            \marginpar{\vspace*{-0.3\baselineskip}
                       \tt\small{#1}\\[-0.4\baselineskip]\rule{\wdtha}{.5pt} }
            \setlength{\marginparwidth}{\wdthc} 
            \setlength{\marginparsep}{\wdthd}  }


\title[CW-resolutions are supported on face posets]{CW-resolutions of monomial ideals that are supported on face posets}
\author[D. Wood]{Daniel Wood} 
\address{Department of Mathematics\\
         University at Albany, SUNY\\ 
         Albany, NY 12222}
\email{dwood@albany.edu}
\keywords{} 
\subjclass{} 

\begin{abstract}
Given a monomial ideal $I$ with minimal free resolution $\mathcal{F}$ supported in characteristic 
$p>0$ on a CW-complex $X$ with regular $2$-skeleton, we construct a CW-complex $Y$ that also 
supports~$\mathcal{F}$ and such that the 
face poset $P(Y)$ also supports $\mathcal{F}$ in the sense of Clark and Tchernev. 
\end{abstract}

\maketitle 

\section*{Introduction} 
Since the work of Taylor~\cite{Ta}, it has been an open problem in commutative algebra to describe explicitly minimal free resolutions of monomial ideals. One way in which 
this problem has classically been approached is to propose a geometric object which models the minimal free resolution $\mathcal{F}$. 
Originally, simplicial complexes were proposed as a model \cite{BaPeSt}. 
Some more general geometric objects proposed were CW-complexes and their cellular chain complexes, which were studied by Bayer and Sturmfels~\cite{C} and Batzies and Welker~\cite{BatWe}. However, 
Velasco~\cite{D} has shown that there are monomial ideals with resolutions that are not supported 
by any CW-complex. Clark and Tchernev~\cite{CT} introduced the notion of resolutions supported on a poset and showed that for every monomial ideal $I$ there is a 
poset $P$ 
that 
supports $\mathcal{F}$. While there are many possible choices for such a poset, $P$, a consequence of Velasco's examples \cite{D} is that, in general, $P$ can not be a face poset of a CW-complex. 
A natural question that arises is to find conditions that allow us to find 
a CW-complex whose face poset supports $\mathcal{F}$.

In the main result of this paper, Theorem~\ref{T:main}, we show that if $I$ has a minimal free resolution $\mathcal{F}$ supported over a field $k$ of characteristic $p>0$ on a CW-complex~$X$ with regular $2$-skeleton, then there is a CW-complex $Y$ that also supports $\mathcal{F}$, and such that the  face poset 
$P(Y)$ of $Y$ over $k$ supports $\mathcal{F}$ as well.



The structure of this paper is as follows. Section 1 outlines the preliminary information necessary for this paper. In Section 2, we discuss deformations of CW-complexes. 
In Section 3, we examine properties of CW-complexes that are homologically isomorphic. In Section 4 we state and prove the main result of the paper. 

I would like to thank both Alexandre Tchernev and Marco Varisco for their helpful discussions and insights. 

\section{Preliminaries}
Fix the following notation: let $k$ be a field of characteristic $p>0$, 
and let $R=k[x_1, \ldots, x_n]$ be a polynomial ring with the standard $\mathbb{Z}^n$-grading. Let 
\[
I = (m_1, \ldots, m_q) \subseteq R
\] 
be a monomial 
ideal of $R$, and let $\mathcal{F}$ be the minimal free resolution of $I$:
\[
\mathcal{F}: 0 \lla F_0 \os{\phi_1}{\lla} F_1 \lla \dotsb
\lla F_{i-1} \os{\phi_i}{\lla} F_i \lla 
\dotsb \lla F_n \lla 0,
\] 
where each $F_i$ is a free $\mathbb{Z}^n$-graded $R$-module, i.e.,
 $F_i = \bigoplus{R(-\alpha)^{\beta_{i,\alpha}}}$ for each $i$ and the differentials 
$\phi_i$ preserve degrees.

Let $X$ be a CW-complex. 
The cellular chain complex of $X$ is the complex $C(X;k)$ of relative homology groups
\begin{multline*}
\ldots \lra H_{n}(X^n,X^{n-1};k) \stackrel{\partial_n}{\lra} 
            H_{n-1}(X^{n-1},X^{n-2};k) \lra 
						\ldots\\
						\ldots \lra H_{1}(X^1,X^{0};k) 
						\stackrel{\partial_1}{\lra} H_{0}(X^0,X^{-1};k) \lra 0,
\end{multline*}
where $X^{-1}$ is the empty set and $X^n$ is the $n$-skeleton of the CW-complex $X$. Let $\partial_0 : 
H_{0}(X^0,X^{-1};k) \lra coker(\partial_1)$ be the canonical projection. 

\begin{remark}
Unless otherwise stated, we will take the basis of $H_{i}(X^{i}, X^{i-1};k)$ to be the set $B_i$ of classes of characteristic maps $\Phi: D^{i} \lra X^{i}$ 
of $i$-cells of $X$, and we set $B = \coprod_{i \geq 0}{B_i}$. We will call $B$ the 
\emph{standard} basis of $C(X;k)$.
\end{remark}

Using these bases, for each $i \geq 1$ we can write the map $\partial_i$ as
\[
\partial_{i}(\sigma) = \sum_{\tau \in B_{i-1}} [\sigma : \tau] \tau
\]
where $[\sigma: \tau]$ is called the \emph{incidence coefficient} of the cells $\sigma$ and $\tau$ over the field $k$.

Consider $B$ ordered as follows: 
for $\sigma$ an $n$-cell and for $\tau$ an $(n-1)$-cell, we will say that 
$\tau \lessdot \sigma \iff [\sigma : \tau] \neq 0 \in k$. 
This leads to a poset, where we define 
\[ \gamma < \sigma \iff \exists \gamma_0 , \ldots, \gamma_m \] with
\[ \gamma = \gamma_0 \lessdot \gamma_1 \lessdot \ldots \lessdot \gamma_m = \sigma \]
which we will call the \emph{incidence poset} or \emph{face poset} over $k$ of the CW-complex $X$ over the field $k$, and 
will denote it with $P$, or $P(X)$ or $P(X;k)$ if the context requires more clarity.
We say that $X$ is a \emph{graded} CW-complex if there is a map 
$mdeg:P(X) \longrightarrow \mathbb{Z}^m$ which is order preserving.
Recall \cite{C, BatWe} that 
$\mathcal{F}$ is \emph{supported on a CW-complex} if there exists a graded CW-complex 
				$X$ that satisifies the following properties:
				\begin{enumerate}
				\item For each $i$, there is a bijection $\eta$ between the basis of $F_i$ and the $i$-cells 
				of~$X$.
				
				\item For each basis element $\sigma$ of $F_i$, we have that $deg(\sigma) = 
				mdeg(\eta(\sigma))$.
				
				\item $\phi_{i}(\sigma)=\sum_{\tau}[\eta(\sigma):\eta(\tau)]x^{mdeg(\eta(\sigma))-mdeg(\eta(\tau))}\tau$.

			\end{enumerate}

%

We also recall the following notions from \cite{CT}. Let $\mathscr{G}$ be an acyclic chain complex 
of free $R$-modules $G_n$.
For each $i$ let $A_i$ be any  basis of $G_i$, and let $A=\coprod{A_i}$. Let $z = \sum_{\tau \in A_i}c_{\tau}\tau$ be an element of $G_i$. The \emph{support} of $z$ with respect to the basis $A$ is the set 
\[
supp_{A}(z) = \left\{ \tau \in A_i : c_\tau \neq 0 \right\}.
\]
When $z \in G_{i}$, we say $z$ is of \emph{minimal support} if it is the case that $z \in im(\partial_{i+1})$ and there is no $0 \neq y \in ker(\partial_i)$ such that 
$supp_{A}(y) \subsetneq supp_{A}(z)$.
When $z \in G_i$ with $i \geq 1$, $supp_{A}(\partial_{i}(z))$ is called the 
\emph{boundary support} of the element $z$. The basis $A$ is called a \emph{basis of minimal (boundary) support} if for each $i \geq 1$ 
and for each $z \in A_i$, the element $\partial_{i}(z) \in G_{i-1}$ is of 
minimal support relative to the basis $A$.

The main result of this paper will deal with monomial ideals $I$ for which $\mathcal{F}$ is 
\emph{supported on a poset}. For the full definition of resolutions supported on a poset, we refer the reader to 
\cite[Definition 2.6]{CT}. We also need the definition of the incidence poset.

\begin{definition}\cite[Definition 4.1]{CT}
Suppose that $\mathscr{G}$ is a chain complex of free $R$-modules $G_n$. 
Let $A_n$ be the basis of $G_n$. Denote by $P(\mathscr{G},A)$ the poset structure on 
$A=\coprod_{n \geq 0}{A_n}$ with ordering given by transitivity applied to the covers: for $x \in A_n$ and $y \in A_{n-1}$, we will have that $y \lessdot x$ if and only if $[x:y] \neq 0 \in k$, where $\partial_{n}(x) = \sum_{z \in A_{n-1}}[x:z]z$. Then we will call $P(\mathscr{G},A)$ the \emph{incidence poset of $\mathscr{G}$} over~$k$ with respect to the basis $A$. 
\end{definition}

We will need the following sufficient condition for when $\mathcal{F}$ is supported on a poset.

\begin{theorem}\label{T:AlexTim47}\cite[Theorem 4.7]{CT}
Let $k$ be a field, let $R = k[x_1, \ldots, x_n]$ be a polynomial ring over $k$, let $I$ be a 
monomial ideal in $R$, and let $\mathcal{F}$ be a minimal $\mathbb{Z}^n$-graded free resolution of $I$ 
over $R$. Suppose that $A$ is a homogenous basis of $\mathcal{F}$ with minimal support, and let 
$deg: P(\mathcal{F},A) \lra \mathbb{Z}^n$ be the map that assigns to each element of $A$ its $\mathbb{Z}^n$-degree 
as an element of $\mathcal{F}$.

Then $deg$ is a morphism of posets, and $P(\mathcal{F},A)$ supports $\mathcal{F}$.
\end{theorem}

It is in the context of Theorem~\ref{T:AlexTim47} that we state our main result, Theorem~\ref{T:main}.

\section{Homology isomorphisms between CW-complexes}
In this section, we establish a convenient fact that allows us to select another CW-complex supporting 
$\mathcal{F}$ with a different selection of attaching maps.
When we say that a CW-complex $X$ is \emph{regular}, we mean that the closure of every cell 
is homeomorphic to a closed ball of the same dimension. For more on this definition, we refer to \cite[p. 534]{B}. 

Recall that a continuous map $f:X \lra Y$ between CW-complexes $X$ and $Y$ is called \emph{cellular} 
if $f(X^n) \subseteq Y^n$, that is, $f$ carries the $n$-skeleton of $X$ onto that of $Y$ \cite[p.~348-349]{B}.

\begin{lemma}\label{L:two}
Let $X$ be an $n$-dimensional connected CW-complex with $n \geq 3$. Suppose there are $l$ cells of 
dimension $n$. For each $n$-cell $\sigma_{j}$ of $X$, 
suppose $F_{j}:(D^n, S^{n-1}) \lra (X^n, X^{n-1})$ is the characteristic map for this 
cell. 
Suppose we are given a $(n-1)$-dimensional CW-complex $Y^{n-1}$ together with
a matrix 
$A=\left(a_{ij}\right)$ in $GL_{l}(\mathbb{Z})$ and a cellular map $\Phi_{n-1} : X^{n-1} \lra Y^{n-1}$
Then there exists an $n$-dimensional CW-complex $Y$ obtained from $Y^{n-1}$ by attaching $l$ cells 
of dimension $n$, with characteristic maps $G_{i}:(D^n, S^{n-1}) \lra (Y^n, Y^{n-1})$, 
and a cellular map $\Phi_n:X^n \lra Y^n$ which extends $\Phi_{n-1}$ and satisfies
\[
[G_i]=\sum_{j}a_{ij}[\Phi_n \circ F_j] \in \pi_{n}(Y^n, Y^{n-1}).
\]
\end{lemma}
\begin{proof}
Pick a basepoint $z \in X$. By \cite [Proposition 0.18, p.~16]{B}, if $\phi$ and $\psi$ are homotopic attaching maps 
for an $n$-cell, then the spaces obtained by attaching a cell along those maps are homotopy equivalent. 
That is to say, if 
\[ 
\begin{CD}
S^{n-1} @>\phi>> Z \\
@VVV             @VVV \\
D^{n} @>\Phi>> X(\phi) \\
\end{CD}
\]
is a pushout diagram for attaching an $n$-cell to a space $Z$, then the fact that $\phi$ is homotopic 
to  $\psi$ implies that $X(\phi)$ is homotopy equivalent $X(\psi)$. As a result of this, we may assume, since $X$ is connected, that the maps $F_j$ are already based maps. 
For based maps $f,g: S^{n-1} \lra Z$, the definition of the map $f+g : S^{n-1} \lra Z$ is given in \cite [p. 340]{B}.


Let $F_{j}|_{S^{n-1}}=f_j$, let $g_{i} = \sum_{j} a_{ij}(\Phi_{n-1} \circ f_{j})$, and let $Y$ be obtained by attaching 
$n$-cells to $Y^{n-1}$ along the maps $g_{i}$.  
For each $i$,
one has in $\pi_{n-1}(Y^{n-1})$ that
$$
\left[g_{i}\right]=\sum_{j} a_{ij}[\Phi_{n-1} \circ f_{j}].
$$
Now, let $G_i:D^n \lra Y$ be the characteristic map for the $n$-cells of $Y$, thus
\[ 
G_i|_{S^{n-1}}=g_i=\sum_{j} a_{ij}(\phi_{n-1} \circ f_{j}).
\] 
To show the existence of a cellular map $\Phi: X \lra Y$, let $A^{-1}=(b_{ij})$ and let
\[
H_j = \sum_{i}b_{ij}G_i.
\]
Therefore $H_{j}|_{S^{n-1}} = h_j = \sum_{j}b_{ij}g_i$ and so in $\pi_{n-1}(Y^{n-1})$ we get 
\[
[h_j]=\sum_{i}b_{ij}[g_j]=\sum_{i}b_{ij}\sum_{k}a_{ik}[\Phi_{n-1} \circ f_k]=[\Phi_{n-1} \circ f_j].
\]
By the homotopy extension property, there exists $H'_i:(D^n, S^{n-1}) \lra (Y^n, Y^{n-1})$ such that 
$[H_i]=[H'_i]$ in $\pi_{n}(Y^n, Y^{n-1})$ and $H'_{i}|_{S^{n-1}}=\Phi_{n-1} \circ f_j$.
By the universal property of pushouts, there exists $\Phi_n : X^n \lra Y^n$ that makes the diagram
\[
\begin{CD}
\coprod{S^{n-1}} @>{\coprod{f_j}}>> X^{n-1} @>{\Phi_{n-1}}>> Y^{n-1} \\
@VVV                @VVV               @VVV \\
\coprod{D^n} @>{\coprod{F_j}}>>   X^n       @>{\Phi_{n}}>> Y^n \\
@|								@.								@| \\
\coprod{D^n} @>{\coprod{H'_j}}>>  Y^n @= Y^n
\end{CD}
\]
commute. Finally we get in $\pi_{n}(Y^n, Y^{n-1})$ the equality
\[
[G_i]=\sum_{j}a_{ij}[H_j] = \sum_{j}a_{ij}[H'_j] = \sum_{j}a_{ij}[\Phi_n \circ F_j]
\]
which completes the proof. 
\end{proof}

\section{Main Results} 
In this section, we state and prove the main result of the paper:

\begin{theorem}\label{T:main}
Let $I$ be a monomial ideal in a polynomial ring $R = k[x_1, \ldots, x_n]$, where $k$ is a fixed field
of characteristic $p>0$, 
and let $X$ be a CW-complex with regular $2$-skeleton such that $X$ supports the minimal free resolution $\mathcal{F}$ of $I$. Then there exists a CW-complex $Y$ such that $Y$ supports $\mathcal{F}$, and the incidence poset $P(Y)$ of 
$Y$ over $k$ supports $\mathcal{F}$ in the sense of \cite{CT}. 

\end{theorem}

Before we prove this, we will need the following consequence of Lemma 3.5 and its proof from \cite{CT} about resolutions supported on posets. For details about the proof we refer the reader there. 


\begin{lemma}\label{L:AlexTim}\cite[Lemma 3.5 and its proof]{CT}
Suppose $R=k[x_1, \ldots, x_m]$ is a polynomial ring over a field $k$ with the standard 
$\mathbb{Z}^m$-grading, and $\mathcal{F}$ is a $\mathbb{Z}^m$-graded free resolution of a torsion-free 
$\mathbb{Z}^m$-graded $R$-module $M$. Let $C_0$ be any basis of $F_0$ consisting of homogenous elements, 
and for $i=1,2$ let $C_i$ be a homogeneous basis of minimal support for $F_i$.

Then $\mathcal{F}$ has a basis with minimal support $A$ consisting of homogenous elements, and such that 
$A_i = C_i$ for $i=0,1,2$.

\end{lemma}

Now we come to a lemma that is a key step in the proof of our main theorem. 

\begin{lemma}\label{T:four}
Let $X$ be a finite dimensional connected CW-complex. 
For each $i$, let $\left\{F^{i}_{\lambda}\right\}_{\lambda \in \Lambda_i}$
be the set of characteristic maps of the $i$-cells of $X$. 
For each $i \geq 3$, let $A_i=(m^{(i)}_{\omega, \lambda})$ be an invertible matrix over $\mathbb{Z}$ of size $|\Lambda_i|=l_i$.
Then there exists a finite dimensional connected CW-complex $Y$ 
and a cellular map $\phi: X \lra Y$ such that 
the set of characteristic maps 
$\left\{G^{i}_{\lambda}\right\}_{\lambda \in \Lambda_i}$ of the $i$-cells of $Y$
satisfies
$$
\left[G^{i}_{\omega}\right]=\sum{m_{\omega, \lambda}\ \left[\phi_ \circ F^{i}_{\lambda}\right]}
\in \pi_{i}(Y^{i}, Y^{i-1}) 
$$
for $i \geq 3$.
\end{lemma}

\begin{proof}
We will construct a CW-complex $Y^n$ inductively as follows. 
Let $Y^i = X^i$ and let $\phi_i = id_{X^i}$ for $i=0,1,2$.
Now we proceed by induction. Assuming for $n \geq 3$ that we have defined $Y^{n-1}$, we seek to define the CW-complex $Y^n$. 
Given the matrix $A_n = (m^{(n)}_{i,j})$, we apply Lemma \ref{L:two} to obtain $Y^n$ and the map $\phi_n:X^n \lra Y^n$ with the desired properties.
\end{proof}

We will also need the following immediate consequence of \cite[Theorem 4.5]{CT}.

\begin{lemma}\label{L:reg}
Let $X$ be a regular $2$-dimensional CW-complex. Then the standard basis of $\widetilde{C}(X;k)$ is 
a basis of minimal support. 
\end{lemma}

Finally, we need the following:

\begin{corollary}\label{C:last}
Let 
$k = \mathbb{Z}/p\mathbb{Z}$ with $p>0$ prime, 
let $X$ be a CW-complex with regular $2$-skeleton, and let
$$
C(X;k): \ldots \os{\partial_{n+1}}{\lra} C_n
\os{\partial_{n}}{\lra} C_{n-1} \os{\partial_{n-1}}{\lra} \dots
\os{\partial_{3}}{\lra} C_2 \os{\partial_{2}}{\lra} C_1 
\os{\partial_{1}}{\lra} C_0 \lra 0
$$ 
be the cellular chain complex for $X$ over $k$. Let $E$ be a basis of minimal support 
of $C(X;k)$ such that $E_i$ is the standard basis of $C_i$ for $i=0,1,2$. 

Then there exists a CW-complex $Y$ and a cellular map $\phi: X \lra Y$ such that $\phi_* :C(X;k) \lra C(Y;k)$ maps $E$ onto the standard basis given by the classes 
of the characteristic maps of $Y$. 




\end{corollary}

\begin{proof}

For each $i \geq 3$ 
the matrix $\overline{M_i}$ that carries the standard basis $B_i$ of $C_i$ onto $E_i$ is invertible over $k$. 
It is well known, see, e.g. \cite[Lemma 6.3.10, p. 347]{Co}, that the natural map $GL_l(\mathbb{Z}) \lra GL_l(\mathbb{Z}/p\mathbb{Z})$ is a surjection, and let $M_i \in GL_l(\mathbb{Z})$. be a matrix that maps modulo $p$ to the matrix $\overline{M_i}$.

Let $A_i = M_i^{-1}$. Now apply Theorem~\ref{T:four} to $X$ and the matrices $A_i$, producing
a CW-complex $Y$ and a cellular map $\Phi: X \lra Y$ as desired.
\end{proof}

With all this, we can now prove our main result:

\begin{proof}[Proof of Theorem 3.1]
Since $X$ supports $\mathcal{F}$, we may assume that 
$k=\mathbb{Z}/p\mathbb{Z}$. By Lemma~\ref{L:reg}, the standard 
homogeneous basis of $\mathcal{F}$ coming from the cells of dimension less than or equal to $2$ is 
of minimal support, and by Lemma~\ref{L:AlexTim} this extends to a homogeneous basis $A$ of minimal support for 
$\mathcal{F}$. By Theorem~\ref{T:AlexTim47}, the incidence poset $P(\mathcal{F},A)$ supports $\mathcal{F}$. Dehomogenizing $A$ yields a basis $E$ of minimal support 
for $C(X;k)$. 
By Corollary~\ref{C:last} we get $Y$ with a face poset $P(Y)$ over $k$ that coincides with the poset 
$P(\mathcal{F},A)$.
\end{proof}

\end{document}